\documentclass[english]{elsarticle}

\usepackage[T1]{fontenc}
\usepackage[latin9]{inputenc}
\usepackage{fancybox}
\usepackage{calc}
\usepackage{units}
\usepackage{amsmath}
\usepackage{amsthm}
\usepackage{amssymb}
\usepackage[all]{xy}

\makeatletter

\numberwithin{equation}{section}
\numberwithin{figure}{section}
  \theoremstyle{definition}
  \newtheorem{defn}{\protect\definitionname}
 \theoremstyle{definition}
  \newtheorem{example}{\protect\examplename}
  \theoremstyle{plain}
  \newtheorem{lem}{\protect\lemmaname}

\usepackage[all]{xy} 
\usepackage{etex}
\xyoption{2cell}
\usepackage{afterpage}
\usepackage{pdflscape}
\usepackage{multicol}
\usepackage[pdftex]{hyperref}
\usepackage{fancyhdr}
\usepackage{setspace}
\usepackage{tikz}
\usetikzlibrary{matrix,arrows}
\usepackage{proof}
\date{}
    \makeatletter
    \def\ps@pprintTitle{%
      \let\@oddhead\@empty
      \let\@evenhead\@empty
      \let\@oddfoot\@empty
      \let\@evenfoot\@oddfoot
    }
    \makeatother

\makeatother

\usepackage{babel}
  \providecommand{\definitionname}{Definition}
  \providecommand{\examplename}{Example}
  \providecommand{\lemmaname}{Lemma}

\begin{document}

\begin{frontmatter}{}

\title{\textbf{\huge{}A Geometry of Multimodal Systems}}

\author[rvt]{Joaqu\'in D\'iaz Boils\corref{cor1}}
 \cortext[cor1]{corresponding author at: boils@uji.es}  \address[rvt]{Facultad de Ciencias Exactas y Naturales.

Pontificia Universidad Cat\'olica del Ecuador. 

170150. Quito. Ecuador.}
\begin{abstract}
{\normalsize{}Multimodal normal incestual systems are investigated
in terms of multiple categories. The different sorted composition
of operators are exhibited as 2-cells in multiple categories built
up from 2-categories giving rise to different axioms. Subsequently,
coherence results are proved pointing the connections with (usual
and mixed) Distributive Laws. This is given as a geometrical description
of certain axioms inside various systems with a number of necessity
and possibility operators.}{\normalsize \par}
\end{abstract}

\end{frontmatter}{}

\section{Introduction}

In this paper we face a categorical perspective about axioms in Multimodal
Logic systems, that is, those dealing with a number of modalities.
We give a characterization of some fragments of the known as \emph{Geach
}and\emph{ McKinsey axiom}s. 

Following previous studies on (uni- or bi-) modal systems, we make
use of monads and comonads as modal operators with the addition of
a number of Distributive Laws. We must consider not just the interaction
between several different monads and between different monads but
also the interaction between both monads and comonads by using four
different kinds of Distributive Laws (two \emph{normal} and two \emph{entwining}).
Particularly, we work with comonads as \emph{necessity operators}
$\square$ and monads as \emph{possibility operators} $\lozenge$
(see \cite{Bierman-de Paiva,Kobayashi}) together with certain natural
transformations allowing the construction of composed modalities as
a new modal operator. These natural transformations are \emph{Distributive
Laws for} $\square$, \emph{Distributive Laws for} $\lozenge$, \emph{Mixed
Distributive Laws from a comonad $\square$ to a monad $\lozenge$}
and\emph{ Mixed Distributive Laws from a monad $\lozenge$ to a comonad
$\square$} in the form respectively:
\[
\square_{a}\square_{b}\rightarrow\square_{b}\square_{a}\qquad\lozenge_{a}\lozenge_{b}\rightarrow\lozenge_{b}\lozenge_{a}\qquad\lozenge_{a}\square_{b}\rightarrow\square_{b}\lozenge_{a}\qquad\square_{a}\lozenge_{b}\rightarrow\lozenge_{b}\square_{a}
\]

The studies quoted above make use of a propositional language performing
conjunction, disjunction and implication. In the spirit of \cite{Dosen-Petric},
and since we are only interested in the interaction among modalities,
we do not mention nor the ambient category over which these modalities
are defined (namely, a bicartesian closed category) neither the way
in which the different endofunctors expressing modalities preserve
that bicartesian closed structure (see \cite{Bierman-de Paiva,Kobayashi}
for these matters). On the other hand, we do not consider negation
to occur in a modality. 

In \cite{Dosen-Petric} a proof-theoretical approach is considered
to introduce categories whose objects are the same modalities and
the deductions (arrows) deal with these modalities. Our study is quite
similar since we consider 2-categories where the 1-cells are monads,
comonads or both alternating and the 2-cells are deductions. Relying
on \cite{Grandis-Par=0000E9}, we consider multiple categories of
cubical type whose 1-cells are, respectively, monads and comonads
for necessity and possibility operators while 2-cells are Distributive
Laws in two different forms: those living in $Mnd(Mnd(\mathcal{C}))$
and those living in $Cmd(Cmd(\mathcal{C}))$. The \emph{category of
quintets} appear as an inspiring example because of its form. We then
proceed to add different sorted arrows in new directions which are
identified with the different modalities of the system. This gives
a new point of view of the well known 3-categories of Distributive
Laws for monads and comonads, obtaining new categories in multiple
form for Distributive Laws, namely $DMnd$ and $DCmd$.

Multiple categories of cubical type and its symmetric variant are
found to be an appropriate setting to obtain a description of Multimodal
systems, in a geometric fashion, based on sets of axioms such as those
in Geach or McKinsey form. For that we consider two applications of
Ehresmann's \emph{category of quintets}, a 2-categorical instance
of double category from which we build up certain \emph{multiple categories
of modalities} by adding more axis to the cells. They seem to be well
suited for Modal systems to express interaction among the different
modalities.

The interaction of possibility and necessity is more complicated and
requires defining carefully the directions to ensure full interaction
systems. We define in Section 6 $Ent$ as the analogous of the 3-categories
of Mixed Distributive or Entwining Laws here in a multiple cubical
way.

Following this line we obtain a number of axioms for Multimodal systems
restricted to an order in the indexing. That is, Distributive Laws
$d_{ij}^{\mathcal{M}}$ for which $i\geq j$. In particular, we obtain
the known as \emph{Persistency axiom} for $\mathcal{M}=\square,\lozenge$
with the above-mentioned constraint. To develop a wider set of axioms
we introduce some transposition functions (2-cycles in the permutation
group) to make all axis permute and allow all modalities interact.
That is, we need to consider a \emph{symmetric version} of $DMnd$
and $DCmd$, increasing drastically the number of axioms available
by permuting all axis in the cubes of $DMnd$ and $DCmd$. Subsequently,
we do the same in Section 6 for $Ent$. For all multiple categories
defined we prove a \emph{Coherence Lemma} stating that the construction
made is consistent, that is, that the endofunctors living in them
behave as expected.

Section 2 offers a brief review of the Multimodal systems considered
in the sequel. In Section 3 a multiple category made of comonads and
Distributive Laws between them is introduced, based on the Ehresmann's
\emph{category of quintets}, to perform systems with necessity operators.
In this Section a knowledge of these concepts is suposed (the contents
about comonads and Distributive Laws can be found in \cite{Street}
and those about multiple categories in \cite{Grandis-Par=0000E9}).
The symmetric counterpart is defined in Section 4, its introduction
is justified by enlarging the set of axioms at our disposal in the
non-symmetric setting. Section 5 has an analogous content than that
of 3 and 4 but based on monads for possibility operators. In Section
6 some interaction Laws are introduced for monads and comonads in
two different ways (relying on \cite{Power-Watanabe} for this matter)
giving rise to the category of Distributive Laws \emph{in all forms}:
$Ent$. Finally, the symmetric version of $Ent$ is given as that
setting including the greatest number of axioms. In section 7 a summary
of the axioms (in the terminology of \cite{Baldoni}) obtained from
every 2-category is given. 

\section{Inclusion Modal Logics}

We introduce some axioms that give rise to multimodal systems generalizing
many existing temporal, dynamic and epistemic modal systems. 
\begin{defn}
A multimodal system is said to be \emph{non-homogeneous} if not all
modal operators belong to the same system.
\end{defn}
An \emph{interaction axiom }is that axiom producing dependent operators.
In particular:
\begin{defn}
\label{inclusion}A \emph{modal inclusion system} is that characterized
by sets of logical axioms in the form 
\[
\square_{t_{1}}...\square_{t_{n}}A\rightarrow\square_{s_{1}}...\square_{s_{m}}A
\]

\noindent for $n>0,m\geq0.$

\label{normal}A Multimodal system is said to be \emph{normal} if
it satisfies $K_{i}-$formulas in the form $K_{i}:\square_{i}(A\rightarrow B)\rightarrow(\square_{i}A\rightarrow\square_{i}B)$.
\end{defn}
\begin{example}
$K_{n},T_{n},K4_{n},S4_{n}$.

We extend the systems with $K_{i}$ by including axioms in the form
$\square_{t_{1}}...\square_{t_{n}}A\rightarrow\square_{s_{1}}...\square_{s_{m}}A$
for $n>0,m\geq0$ where $t_{1},...t_{n},s_{1},...,s_{m}$ belong to
a certain language which we call $Mod$.

This class of axioms is included into the one defined in \cite{Catach}.
We proceed now to add some symbols to the indexing language $Mod$:
\end{example}
\begin{itemize}
\item a binary operator $\cup$ (\emph{non-deterministic choice})
\item a binary operator $;$ (\emph{sequential composition}) 
\item $\epsilon$ (neutral element for composition).
\end{itemize}
With them we are able to create more labels for modal operators having
also the rule:

\begin{quotation}
\begin{center}
\emph{If $A$ is a proposition in a certain language and $i\in Mod$
then $\square_{i}A$ is also a proposition in that language.} 
\par\end{center}
\end{quotation}
\begin{defn}
\label{incestual}In \cite{Catach} the \emph{incestual modal logic}
is the class of normal modal logic containing the axioms:\footnote{However, in our modelling we will not make use of non-deterministic
choice operator.} 

\[
\square_{\epsilon}A\longleftrightarrow A\qquad\square_{i;j}A\longleftrightarrow\square_{i}\square_{j}A\qquad\square_{i\cup j}A\longleftrightarrow\square_{i}A\wedge\square_{j}A
\]

\end{defn}

In \cite{Baldoni} it is defined for $i,j\in Mod$ a set of \emph{incestual
axioms} by including possibility operators $\lozenge$ in the form\footnote{Baldoni's is included into Catach by $\{G^{\epsilon,b,c,\epsilon}\}\subset\{G^{a,b,c,d}\}$
as seen below.}
\[
G^{a,b,c,d}:\lozenge_{a}\square_{b}A\longrightarrow\square_{c}\lozenge_{d}A
\]
with $a,b,c,d\in Mod$. The axiom $G^{a,b,c,d}$, a generalization
of the known as \emph{Geach Axiom}, is precisely what we take as the
one for which we will construct our model.\footnote{ It should be noticed that axioms in the form $\square_{t_{1}}...\square_{t_{n}}A\rightarrow\square_{s_{1}}...\square_{s_{m}}A$
for $n>0,m\geq0$ where $t_{1},...t_{n},s_{1},...,s_{m}\in Mod$ are
in the form $G^{a,b,c,d}$ by taking $a=\epsilon,b=t_{1};...;t_{n},c=s_{1};...;s_{m},d=\epsilon$
while axioms $\lozenge_{t_{1}}...\lozenge_{t_{n}}A\rightarrow\lozenge_{s_{1}}...\lozenge_{s_{m}}A$
are obtained by taking $a=t_{1};...;t_{n},b=\epsilon,c=\epsilon,d=s_{1};...;s_{m}$.}

This paper develops from a categorical point of view certain fragments
of \emph{generalized Geach axioms} in the forms $G^{\epsilon,b,c,\epsilon}$
(with no necessity operators), $G^{a,\epsilon,\epsilon,d}$ (with
no possibility operators) and $G^{a,b,b,a}$ as well as axioms $\square_{a}\lozenge_{b}A\rightarrow\lozenge_{b}\square_{a}A$,
which are in \emph{McKinsey form}.

\section{The cubical categories $DCmd(\mathcal{C})$}

We make use of the concept of monad and comonad for possibility and
necessity operators respectively as in \cite{Dosen-Petric}. However,
we need to perform the (many different) interaction rules between
the comonads which are involved in every axiom. So an extensive use
of \emph{Distributive Laws for monads} and\emph{ Distributive Laws
for comonads} has to be considered in the context of $Mnd(\mathcal{C})$
and $Cmd(\mathcal{C})$, the 2-categories of monads and comonads.
For that, rather than considering the 3-categories of Distributive
Laws for monads and comonads over a 2-category, namely $Mnd(Mnd(\mathcal{C}))$
and $Cmd(Cmd(\mathcal{C}))=Mnd(Mnd(\mathcal{C}_{op})_{op})$, introduced
in \cite{Street} (see \cite{Chikhladze} for a more recent treatment),
we will make use of the language of multiple categories to consider
coposition of Distributive Laws because of its clarity to illustrate
the interaction between modalities in a geometric fashion.

As known we can compose monads and comonads separatedly whenever we
have a \emph{Distributive Law} at our disposal. We show in this section
which the form of a multiple composition of comonads is for the \emph{possibility}
operators. In the 2-categorical point of view of \cite{Street} a
Distributive Law in the 2-categorical language appears as an instance
of a comonad functor.

Multiple categories, as introduced in \cite{Grandis-Par=0000E9},
are the generalization of cubical categories for which the arrows
can be thought of having different shape. In our case this will mean
to increase the length in which a certain modal operator appears in
sequences as those referred in Definition \ref{inclusion}. Some sequences
will be the domain of arrows relating the length of a modal composite
with some cells in a certain direction. We will make use of these
sequences in order to give a geometrical point of view about multimodal
interaction.
\begin{defn}
\label{multiplecat}A \emph{multiple category} $\mathcal{C}$ \emph{of
cubical type} is a multiple set of cubical type (see \cite{Grandis-Par=0000E9})
with components $\mathcal{C}_{\mathbf{i}}$ whose elements are $\mathbf{i}-cells$
subject to the following data:

\begin{enumerate}
\item given two composable $\mathbf{i}-cells$ $a,b$ (that is: such that
$\partial_{i}^{+}(a)=\partial_{i}^{-}(b)$ for $i\in\mathbf{i}$)
we have the $i-composition$, denoted by $a+_{i}b$, satisfying:{\footnotesize{}
\[
\begin{array}{c}
\partial_{i}^{-}(a+_{i}b)=\partial_{i}^{-}(a)\qquad\partial_{i}^{+}(a+_{i}b)=\partial_{i}^{+}(b)\\
\partial_{j}^{\alpha}(a+_{i}b)=\partial_{j}^{\alpha}(a)+_{i}\partial_{j}^{\alpha}(b)\qquad e_{j}(a+_{i}b)=e_{j}(a)+_{i}e_{j}(b)\textrm{ for }i\neq j
\end{array}
\]
}{\footnotesize \par}
\item for $j\notin\mathbf{i}$ we have a new category with objects in $\mathcal{C}_{\mathbf{i}}$,
arrows in $\mathcal{C}_{\mathbf{i}j}$ (standing for $\mathcal{C}_{\mathbf{i}\cup j}$)
and new faces $\partial_{j}^{\alpha}$, identities $e_{j}$ and composition
$+_{j}$
\item for $i<j$ we have the \emph{middle four-interchange }rule{\footnotesize{}
\[
(a+_{i}b)+_{j}(c+_{i}d)=(a+_{j}c)+_{i}(b+_{j}d)
\]
}{\footnotesize \par}
\end{enumerate}
\end{defn}
An $\mathbf{n}-cell$ in a multiple category of cubical type will
then be called an $n-cube$ whenever $\mathbf{n}$ is the $n^{th}$
ordinal $\{0,...,n-1\}$. In the following sections all multiple categories
are with no mention of cubical type.

We consider double categories over 2-categories. There exist several
different ways to get a double category from a 2-category (see \cite{Fiore}),
our approach takes the \emph{double category of quintets} $Q\mathcal{C}$
of Ehresmann as an inspiring example. By endowing a double category,
in Ehresmann's form, with $n$ arrows in different directions which
turn out to be comonads, we will construct an n-dimensional multiple
category of cubical type. 
\begin{defn}
Given a 2-category $\mathcal{C}$ we denote by $Q\mathcal{C}$ that
double category whose objects are those of $\mathcal{C}$, whose horizontal
and vertical arrows are the 1-cells of $\mathcal{C}$ and whose squares
are the 2-cells $a:K\circ F\rightarrow G\circ J$ in{\footnotesize{}
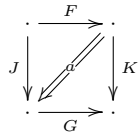
\[
\xymatrix{\cdot\ar[r]^{F}\ar[d]_{J} & \cdot\ar[d]^{K}\ar@{=>}[dl]|{a}\\
\cdot\ar[r]_{G} & \cdot
}
\]
}{\footnotesize \par}
\end{defn}
It is precisely from $Q\mathcal{C}$ that we take the point of view
of a double cell (a square) as a 2-cell with the same 2-categorical
object in all four nodes, obtaining the 2-categorical description
of Distributive Laws between comonads from Section 3. That is, given
a 2-category $\mathcal{C}$ we consider a double category for which
the diagonal 2-cells are Distributive Laws for comonads in the form
$a:N_{2}N_{1}\rightarrow N_{1}N_{2}$, where we identify $F$ and
$G$ with $N_{1}$ and $V$ and $V$ with $N_{2}$. 
\begin{defn}
Let $DCmd_{\mathbf{2}}(\mathcal{C})$ be the full subcategory of $Q\mathcal{C}$
for which all 1-cells are comonads and all diagonal 2-cells are Distributive
Laws between them. 
\end{defn}
$DCmd_{\mathbf{2}}(\mathcal{C})$, as a double category, is endowed
with horizontal maps in a square{\footnotesize{}
\[
\xymatrix{N_{2}N_{1}C\ar[r]^{aC}\ar[d]_{N_{2}N_{1}u} & N_{1}N_{2}C\ar[d]^{N_{1}N_{2}u}\ar@{=>}[dl]|{au}\\
N_{2}N_{1}C\ar[r]_{aC} & N_{1}N_{2}C
}
\]
}for $C$ an object and $u:C\rightarrow C$ a 1-cell in $\mathcal{C}$.

From that double category we describe a multiple category $DCmd_{\mathbf{n}}(\mathcal{C})$
of Distributive Laws between $n$ comonads by endowing $DCmd_{\mathbf{2}}(\mathcal{C})$
with axis in more directions, as done in \cite{Grandis-Par=0000E9}
Those axis will play the rol of modalities acting over the objects
of $\mathcal{C}$, for that we identify the directions of the axis
with the indexing of the modalities. 

Let us remark that, while $Cmd(Cmd(\mathcal{C}_{op})_{op})$ from
\cite{Street} is a 2-category of Distributive Laws between comonads,
$DCmd_{\mathbf{n}}(\mathcal{C})$ is a multiple category containing
a geometrical (cubical) account of the different compositions of Distributive
Laws that can be considered. In fact, $DCmd_{\mathbf{2}}(\mathcal{C})$
can be seen as the image of $Cmd(Cmd(\mathcal{C}_{op})_{op})$ through
the known as \emph{functor of quintets} 
\[
Q:2-Cat\longrightarrow Dbl
\]
where $Dbl$ denotes the category of all double categories.

We now give a description of $DCmd_{\mathbf{3}}(\mathcal{C})$. Let
$\mathcal{C}$ be a 2-category and $C$ an object in $\mathcal{C}$:\footnote{Although it is not our concern in this paper, $\mathcal{C}$ should
have a bicartesian closed underlying category to get conjunctions,
disjunctions and implications between the objects of $\mathcal{C}$
as well as the requirement that all comonads are \emph{symmetric monoidal
closed} (see \cite{Bierman-de Paiva} for this matter).}
\begin{enumerate}
\item $DCmd_{\emptyset}(\mathcal{C})$ is the category whose objects are
the objects of $\mathcal{C}$
\item $DCmd_{0}(\mathcal{C}),DCmd_{1}(\mathcal{C}),DCmd_{2}(\mathcal{C})$
are the categories whose objects are comonads over $C$ in the horizontal
($\square_{0}$), diagonal ($\square_{1}$) and vertical ($\square_{2}$)
directions respectively endowed with one degeneracy and two faces
for $i=0,1,2$

\begin{center}
$e_{i}:DCmd_{\emptyset}(\mathcal{C})\rightarrow DCmd_{i}(\mathcal{C})\qquad\partial_{i}^{\alpha}:DCmd_{i}(\mathcal{C})\rightarrow DCmd_{\emptyset}(\mathcal{C})$
\par\end{center}
\item $DCmd_{01}(\mathcal{C}),DCmd_{12}(\mathcal{C}),DCmd_{02}(\mathcal{C})$
are the categories whose objects are , from left to right, squares
together with three 2-cells $\ensuremath{d_{10}^{\square}}:\square_{1}\square_{0}\rightarrow\square_{0}\square_{1},d_{21}^{\square}:\square_{2}\square_{1}\rightarrow\square_{1}\square_{2}$
and $d_{20}^{\square}:\square_{2}\square_{0}\rightarrow\square_{0}\square_{2}$

{\footnotesize{}
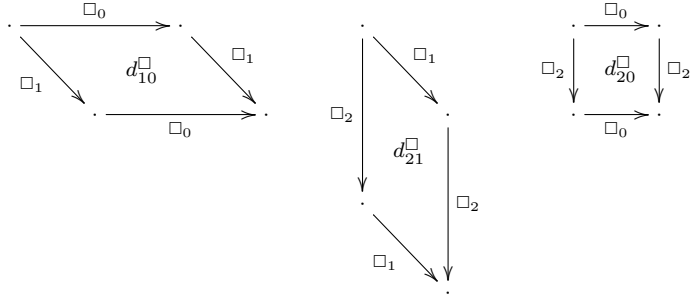
\[
\xymatrix{\cdot\ar[rr]^{\square_{0}}\ar[dr]_{\square_{1}}\ar@{}[drrr]|{\mathstrut\raisebox{0.2ex}{\textrm{ }\textrm{ }\ensuremath{d_{10}^{\square}}\kern0.3em }} &  & \cdot\ar[dr]^{\square_{1}}\\
 & \cdot\ar[rr]_{\square_{0}} &  & \cdot
}
\qquad\xymatrix{\cdot\ar[rd]^{\square_{1}}\ar[dd]_{\square_{2}}\ar@{}[dddr]|{\mathstrut\raisebox{2ex}{\textrm{ }\textrm{ }\ensuremath{d_{21}^{\square}}\kern0.3em }}\\
 & \cdot\ar[dd]^{\square_{2}}\\
\cdot\ar[rd]_{\square_{1}}\\
 & \cdot
}
\qquad\xymatrix{\cdot\ar[r]^{\square_{0}}\ar[d]_{\square_{2}}\ar@{}[dr]|{\mathstrut\raisebox{0.5ex}{\textrm{ }\textrm{ }\ensuremath{d_{20}^{\square}}\kern0.3em }} & \cdot\ar[d]^{\square_{2}}\\
\cdot\ar[r]_{\square_{0}} & \cdot
}
\]
}respectively behaving as Distributive Laws according to the definition
given in Section 4 and two degeneracies and four faces 

\begin{center}
$e_{i}:DCmd_{j}(\mathcal{C})\rightarrow DCmd_{ij}(\mathcal{C})\qquad e_{j}:DCmd_{i}(\mathcal{C})\rightarrow DCmd_{ij}(\mathcal{C})\qquad\partial_{i}^{\alpha}:DCmd_{ij}(\mathcal{C})\rightarrow DCmd_{j}(\mathcal{C})\qquad\partial_{j}^{\alpha}:DCmd_{ij}(\mathcal{C})\rightarrow DCmd_{i}(\mathcal{C})$
\par\end{center}

\noindent for $i,j=0,1,2$ such that $i<j$ and $\alpha=0,1$.\footnote{Faces and degeneracies act respectively as \emph{erasing} and \emph{introducing}
a new modality in the system.}
\item $DCmd_{012}(\mathcal{C})$ is a category whose objects are 3-cubes
each face of which comes with a diagonal 2-cell in it.

For the comonads $\square_{0},\square_{1},\square_{2}$ defined in
the same object $C$ in the 2-category $\mathcal{C}$, and the directions
{\footnotesize{}
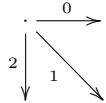
\[
\xymatrix{\cdot\ar[r]^{0}\ar[dr]_{1}\ar[d]_{2} & \mbox{}\\
\mbox{} & \mbox{}
}
\]
}the 3-cells in $\mathcal{C}_{012}$ are cubes as given at left, each
face containing a Distributive Law as given at right:

{\scriptsize{}
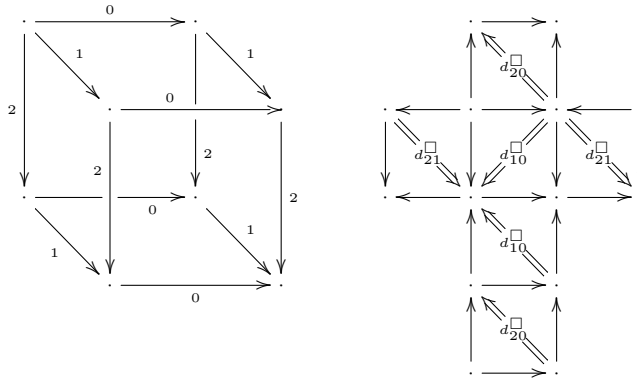
\hspace{3em}\xymatrix{\cdot\ar[rr]^{0}\ar[dd]_{2}\ar[dr]^{1} & \mbox{} & \cdot\ar'[d][dd]^{2}\ar[dr]^{1}\\\mbox{} & \cdot\ar'[rr]^(.35){0}\ar[dd]_(.35){2} &  & \cdot\ar[dd]^{2}\\\cdot\ar'[r][rr]_{0}\ar[dr]_{1} &  & \cdot\ar[dr]^{1}\\ & \cdot\ar[rr]_{0} &  & \cdot} \hspace{3em} \xymatrix{ & \cdot\ar[r] & \cdot\\\cdot\ar@{=>}[dr]|{d_{21}^{\square}}\ar[d] & \cdot\ar[u]\ar[d]\ar[r]\ar[l] & \cdot\ar[u]\ar@{=>}[ul]|{d_{20}^{\square}}\ar@{=>}[dr]|{d_{21}^{\square}}\ar[d]\ar@{=>}[dl]|{d_{10}^{\square}} & \cdot\ar[d]\ar[l]\\\cdot & \cdot\ar[l]\ar[r] & \cdot\ar[r] & \cdot\\ & \cdot\ar[r]\ar[u] & \cdot\ar@{=>}[ul]|{d_{10}^{\square}}\ar[u]\\ & \cdot\ar[u]\ar[r] & \cdot\ar[u]\ar@{=>}[ul]|{d_{20}^{\square}}}}{\scriptsize \par}
\end{enumerate}
{\scriptsize{}}

{\scriptsize \par}Now we describe how $i;j$ subindexes are defined from a cubical set
structure performing \emph{concatenation}. With them we express the
multiple relations between the modalities in cubical form. Every subset
of a multi-index set of ordinals $\mathbf{n}$ can be seen as a subindex
for a comonad. Concatenation of modalities are described by paths
of edges of hypercubes whose dimension is that of the number of modalities
we are dealing with. This is geometrically expressed by chains of
hypercubes for which every kind of arrow refers to a different modality
by being oriented in a different direction.\footnote{In \cite{Dosen-Petric} there is an interpretation of modal logics
with one and two operators in terms of relations between the length
of a composite of these modalities.} 

Our modalities (and the nodes of the sets into a multiple cubical
category) will then be $\square_{i_{1}}...\square_{i_{n}}=\square_{i_{1};...;i_{n}}$
for $i_{1},...,i_{n}\in\mathbf{n}$ such that $i_{1}<...<i_{n}$.\footnote{From a logical perspective, subindexes in the form of a concatenation
(a semicolon chain) describes the path to reach a certain \emph{place
of observation} or \emph{place of knowledge}.}

We now show how can one compose Distributive Laws for comonads over
a 2-category in $DCmd(\mathcal{C})$. These compositions are expressed
here as a concatenation of 2-cubes in a multiple category of cubical
type. We make use of the hypercube notation for compositions as introduced
in \cite{Grandis-Par=0000E9}.

For instance, directed compositions in $DCmd_{\mathbf{3}}(\mathcal{C})$,
given the $0,1,2$ axis as above, are{\footnotesize{}
\[
\begin{array}{c}
d_{10}^{\square}+_{0}d_{10}^{\square}=d_{1(00)}^{\square}\\
d_{10}^{\square}+_{1}d_{10}^{\square}=d_{(11)0}^{\square}\\
d_{20}^{\square}+_{0}d_{20}^{\square}=d_{2(00)}^{\square}
\end{array}\qquad\begin{array}{c}
d_{20}^{\square}+_{2}d_{20}^{\square}=d_{(22)0}^{\square}\\
d_{21}^{\square}+_{1}d_{21}^{\square}=d_{2(11)}^{\square}\\
d_{21}^{\square}+_{2}d_{21}^{\square}=d_{(22)1}^{\square}
\end{array}
\]
}where we express $d_{i(jj)}^{\square}$ for $d_{i;(j;j)}^{\square}$.
In $DCmd_{\mathbf{n}}(\mathcal{C})$ we have, for $i,j\in\mathbf{n}$
such that $i\geq j$:{\footnotesize{}
\[
\begin{array}{c}
d_{ij}^{\square}+_{j}d_{ij}^{\square}=d_{i(jj)}^{\square}\\
d_{ij}^{\square}+_{i}d_{ij}^{\square}=d_{(ii)j}^{\square}
\end{array}
\]
}Expressing $(d\mid d')$ and $\left(\frac{d}{d'}\right)$ for $d+_{0}d'$
and $d+_{1}d'$ respectively in $DCmd_{\mathbf{2}}(\mathcal{C})$
we have for {\footnotesize{}
\[
\xymatrix{\cdot\ar[r]^{N_{1}}\ar[d]_{N_{2}} & \cdot\ar[d]\ar[r]^{N'_{1}}\ar@{=>}[dl]|{d_{N_{2}N_{1}}^{\square}} & \cdot\ar[d]^{N_{2}}\ar@{=>}[dl]|{d_{N_{2}N'_{1}}^{\square}}\\
\cdot\ar[d]_{N'_{2}}\ar[r] & \cdot\ar[d]\ar[r]\ar@{=>}[dl]|{d_{N'_{2}N_{1}}^{\square}} & \cdot\ar[d]^{N_{2}'}\ar@{=>}[dl]|{d_{N'_{2}N'_{1}}^{\square}}\\
\cdot\ar[r]_{N_{1}} & \cdot\ar[r]_{N'_{1}} & \cdot
}
\]
}{\footnotesize \par}

\noindent horizontal and vertical compositions 
\[
(d_{N_{2}N_{1}}^{\square}\mid d_{N_{2}N_{1}'}^{\square})=N_{1}'d_{N_{2}N_{1}}^{\square}\cdot d_{N_{2}N_{1}'}^{\square}N_{1}\qquad\left(\frac{d_{N'_{2}N_{1}}^{\square}}{d_{N_{2}N_{1}}^{\square}}\right)=d_{N_{2}N_{1}}^{\square}N'_{2}\cdot N_{2}d_{N'_{2}N_{1}}^{\square}
\]
both being strict thanks to the 2-categorical structure of $\mathcal{C}$.
Finally, we have the \emph{four-middle interchange rule}:
\[
\left(\frac{d_{N_{2}N_{1}}^{\square}\mid d_{N_{2}N_{1}'}^{\square}}{d_{N'_{2}N_{1}}^{\square}\mid d_{N'_{2}N_{1}'}^{\square}}\right)=\left(\frac{d_{N{}_{2}N_{1}}^{\square}}{d_{N'_{2}N_{1}}^{\square}}\mid\frac{d_{N{}_{2}N'_{1}}^{\square}}{d_{N'_{2}N'_{1}}^{\square}}\right)
\]
for \emph{single Distributive Laws}.\footnote{We also have some singular instances of 2-cubes such as{\footnotesize{}
\[
\xymatrix{\cdot\ar[r]^{N_{i}}\ar[d]_{N_{i}} & \cdot\ar[d]^{N_{i}}\ar@{=>}[dl]|{1_{N_{i}N_{i}}^{\square}}\\
\cdot\ar[r]_{N_{i}} & \cdot
}
\qquad\xymatrix{\cdot\ar[r]^{N_{i}}\ar[d]_{1} & \cdot\ar[d]^{1}\ar@{=>}[dl]|{d_{N_{i}}^{\square}}\\
\cdot\ar[r]_{N_{i}} & \cdot
}
\qquad\xymatrix{\cdot\ar[r]^{1}\ar[d]_{N_{i}} & \cdot\ar[d]^{N_{i}}\ar@{=>}[dl]|{d_{N_{i}}^{\square}}\\
\cdot\ar[r]_{1} & \cdot
}
\]
}They are instances of the so-called \emph{special iso-cells} in \cite{Grandis-Par=0000E9}
and give rise to the $K_{i}$ axioms characterizing the normal multimodal
systems (see Definition \ref{normal}).} 

The following is our first \emph{coherence Lemma} for the interaction
of n necessity modalities. 
\begin{lem}
Every composition of 2-cells of the form $d^{\square}$ in $DCmd_{\mathbf{n}}(\mathcal{C})$
is a Distributive Law. 
\end{lem}
\begin{proof}
Take $i,j,k\in\mathbf{n}$ such that $k\geq j\geq i$. Having already
identities and associativity, for {\footnotesize{}
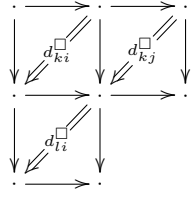
\[
\xymatrix{\cdot\ar[r]\ar[d] & \cdot\ar@{=>}[dl]|{d_{ki}^{\square}}\ar[r]\ar[d] & \cdot\ar[d]\ar@{=>}[dl]|{d_{kj}^{\square}}\\
\cdot\ar[r]\ar[d] & \cdot\ar@{=>}[dl]|{d_{li}^{\square}}\ar[r]\ar[d] & \cdot\\
\cdot\ar[r] & \cdot
}
\]
}we can compose

\begin{enumerate}
\item $(d_{ki}^{\square}\mid d_{kj}^{\square}):\square_{k}(\square_{j}\square_{i})\rightarrow(\square_{j}\square_{i})\square_{k}$
defined as $\square_{j}d_{ki}^{\square}\cdot d_{kj}^{\square}\square_{i}$
and for which we have, according to the 2-categorical definition of
Distributive Law of \cite{Street}, a pair 
\[
((C,\square_{j}\square_{i}),(\square_{k},(d_{kj}^{\square}\mid d_{ki}^{\square})))
\]
This is based on two Distributive Laws $d_{kj}^{\square}$ and $d_{ki}^{\square}$
seen as comonad functors:

\begin{enumerate}
\item as $d_{kj}^{\square}$ and $d_{ki}^{\square}$ are Distributive Laws
we have the following commuting diagrams {\footnotesize{}
\[
\xymatrix{ & \square_{k}\square_{i}\ar[dl]_{\varepsilon_{k}\square_{i}}\ar[dd]^{d_{ki}^{\square}}\\
\square_{i}\\
 & \square_{i}\square_{k}\ar[ul]^{\square_{i}\varepsilon_{k}}
}
\qquad\xymatrix{ & \square_{k}\square_{j}\ar[dl]_{\varepsilon_{k}\square_{j}}\ar[dd]^{d_{kj}^{\square}}\\
\square_{j}\\
 & \square_{j}\square_{k}\ar[ul]^{\square_{j}\varepsilon_{k}}
}
\]
}and then {\footnotesize{}
\[
\xymatrix{ &  & \square_{k}\square_{j}\square_{i}\ar[ddll]_{\varepsilon_{k}\square_{j}\square_{i}}\ar[dd]^{d_{kj}^{\square}\square_{i}}\ar@/^{4pc}/[dddd]^{(d_{kj}^{\square}\mid d_{ki}^{\square})}\\
\\
\square_{j}\square_{i} &  & \square_{j}\square_{k}\square_{i}\ar[ll]_{\square_{j}\varepsilon_{k}\square_{i}}\ar[dd]^{\square_{j}d_{ki}^{\square}}\\
\\
 &  & \square_{j}\square_{i}\square_{k}\ar[lluu]^{\square_{j}\square_{i}\varepsilon_{k}}
}
\]
}commutes. 

\item On the other hand from {\tiny{}
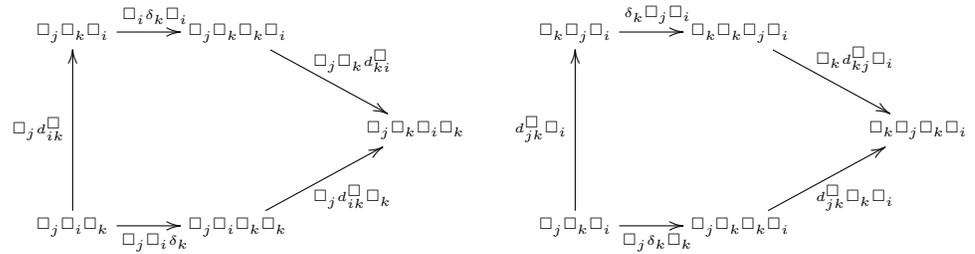
\[
\xymatrix{\square_{j}\square_{k}\square_{i}\ar[r]^{\square_{i}\delta_{k}\square_{i}} & \square_{j}\square_{k}\square_{k}\square_{i}\ar[dr]^{\square_{j}\square_{k}d_{ki}^{\square}}\\
 &  & \square_{j}\square_{k}\square_{i}\square_{k}\\
\square_{j}\square_{i}\square_{k}\ar[uu]^{\square_{j}d_{ik}^{\square}}\ar[r]_{\square_{j}\square_{i}\delta_{k}} & \square_{j}\square_{i}\square_{k}\square_{k}\ar[ru]_{\square_{j}d_{ik}^{\square}\square_{k}}
}
\qquad\xymatrix{\square_{k}\square_{j}\square_{i}\ar[r]^{\delta_{k}\square_{j}\square_{i}} & \square_{k}\square_{k}\square_{j}\square_{i}\ar[dr]^{\square_{k}d_{kj}^{\square}\square_{i}}\\
 &  & \square_{k}\square_{j}\square_{k}\square_{i}\\
\square_{j}\square_{k}\square_{i}\ar[uu]^{d_{jk}^{\square}\square_{i}}\ar[r]_{\square_{j}\delta_{k}\square_{k}} & \square_{j}\square_{k}\square_{k}\square_{i}\ar[ru]_{d_{jk}^{\square}\square_{k}\square_{i}}
}
\]
}standing for $\square_{j}d_{ki}^{\square}$ and $d_{kj}^{\square}\square_{i}$,
we obtain{\footnotesize{}
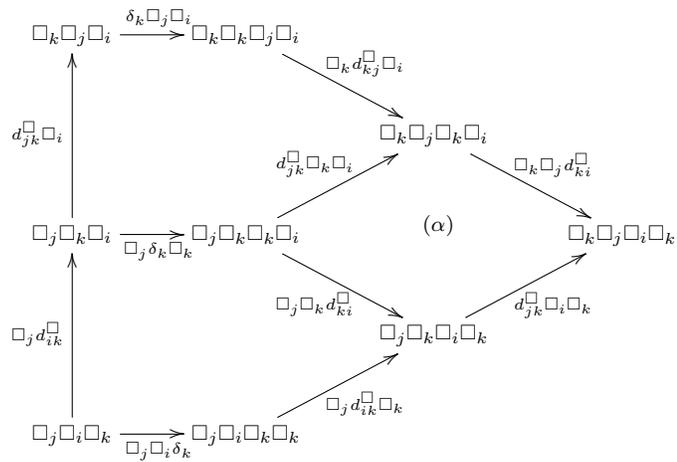
\[
\xymatrix{\square_{k}\square_{j}\square_{i}\ar[r]^{\delta_{k}\square_{j}\square_{i}} & \square_{k}\square_{k}\square_{j}\square_{i}\ar[dr]^{\square_{k}d_{kj}^{\square}\square_{i}}\\
 &  & \square_{k}\square_{j}\square_{k}\square_{i}\ar[dr]^{\square_{k}\square_{j}d_{ki}^{\square}}\ar@{}[dd]|{\mathstrut\raisebox{2ex}{\textrm{ }\textrm{ }\ensuremath{(\alpha)}\kern0.3em }}\\
\square_{j}\square_{k}\square_{i}\ar[uu]^{d_{jk}^{\square}\square_{i}}\ar[r]_{\square_{j}\delta_{k}\square_{k}} & \square_{j}\square_{k}\square_{k}\square_{i}\ar[ru]^{d_{jk}^{\square}\square_{k}\square_{i}}\ar[rd]_{\square_{j}\square_{k}d_{ki}^{\square}} &  & \square_{k}\square_{j}\square_{i}\square_{k}\\
 &  & \square_{j}\square_{k}\square_{i}\square_{k}\ar[ru]_{d_{jk}^{\square}\square_{i}\square_{k}}\\
\square_{j}\square_{i}\square_{k}\ar[uu]^{\square_{j}d_{ik}^{\square}}\ar[r]_{\square_{j}\square_{i}\delta_{k}} & \square_{j}\square_{i}\square_{k}\square_{k}\ar[ru]_{\square_{j}d_{ik}^{\square}\square_{k}}
}
\]
}for the second condition of a comonad functor where the square $(\alpha)$
commutes by naturality of composition.
\item We need to show also that there exists a comonad natural transformation
\[
\varepsilon:(\square_{k},(d_{kj}^{\square}\mid d_{ki}^{\square}))\rightarrow1
\]
from comonad natural transformations $\varepsilon_{k}:(\square_{k},d_{kj}^{\square})\rightarrow1$
and $\varepsilon'_{k}:(\square_{k},d_{ki}^{\square})\rightarrow1$
subject to the commuting squares 
\[
\xymatrix{\square_{k}\square_{j}\ar[r]^{\varepsilon_{k}\square_{j}}\ar[d]_{d_{kj}^{\square}} & \square_{j}\ar[d]^{1}\\
\square_{j}\square_{k}\ar[r]_{\square_{j}\varepsilon_{k}} & \square_{j}
}
\qquad\xymatrix{\square_{k}\square_{i}\ar[r]^{\varepsilon'_{k}\square_{i}}\ar[d]_{d_{ki}^{\square}} & \square_{i}\ar[d]^{1}\\
\square_{i}\square_{k}\ar[r]_{\square_{i}\varepsilon'_{k}} & \square_{i}
}
\]
from which we get 
\[
\xymatrix{\square_{k}\square_{j}\square_{i}\ar[r]^{\varepsilon\square_{j}\square_{i}}\ar[d]_{d_{kj}^{\square}\square_{i}}\ar@/_{4pc}/[dd]_{(d_{kj}^{\square}\mid d_{ki}^{\square})} & \square_{j}\square_{i}\ar[d]^{1}\\
\square_{j}\square_{k}\square_{i}\ar[d]_{\square_{j}d_{ki}^{\square}}\ar[r] & \square_{j}\square_{i}\ar[d]^{1}\\
\square_{j}\square_{i}\square_{k}\ar[r]_{\square_{j}\square_{i}\varepsilon} & \square_{j}\square_{i}
}
\]
commuting for $\varepsilon$.
\item Finally, for duplication $\delta$ we have duplications {\small{}
\[
\delta^{kj}:(\square_{k},d_{kj}^{\square})\rightarrow(\square_{k},d_{kj}^{\square})\cdot(\square_{k},d_{kj}^{\square})\textrm{ and }\delta^{ki}:(\square_{k},d_{ki}^{\square})\rightarrow(\square_{k},d_{ki}^{\square})\cdot(\square_{k},d_{ki}^{\square})
\]
}in the form {\footnotesize{}
\[
\delta^{kj}:(\square_{k},d_{kj}^{\square})\rightarrow(\square_{k}\square_{k},d_{kj}^{\square}\square_{k}\cdot\square_{k}d_{kj}^{\square})\textrm{ and }\delta^{ki}:(\square_{k},d_{ki}^{\square})\rightarrow(\square_{k}\square_{k},d_{ki}^{\square}\square_{k}\cdot\square_{k}d_{ki}^{\square})
\]
}respectively for which 
\[
\xymatrix{\square_{k}\square_{j}\ar[r]^{\delta^{kj}\square_{j}}\ar[d]_{d_{kj}^{\square}} & \square_{k}\square_{k}\square_{j}\ar[d]^{d_{kj}^{\square}\square_{k}\cdot\square_{k}d_{kj}^{\square}}\\
\square_{j}\square_{k}\ar[r]_{\square_{j}\delta^{j}} & \square_{j}\square_{k}\square_{k}
}
\qquad\xymatrix{\square_{k}\square_{i}\ar[r]^{\delta^{ki}\square_{i}}\ar[d]_{d_{ki}^{\square}} & \square_{k}\square_{k}\square_{i}\ar[d]^{d_{ki}^{\square}\square_{k}\cdot\square_{k}d_{ki}^{\square}}\\
\square_{i}\square_{k}\ar[r]_{\square_{i}\delta^{ki}} & \square_{i}\square_{k}\square_{k}
}
\]
commute. From them we get a duplication {\small{}
\[
\delta:(\square_{k},(d_{kj}^{\square}\mid d_{ki}^{\square}))\rightarrow(\square_{k},(d_{kj}^{\square}\mid d_{ki}^{\square}))\cdot(\square_{k},(d_{kj}^{\square}\mid d_{ki}^{\square}))
\]
}in the form {\small{}
\[
\delta:(\square_{k},(d_{kj}^{\square}\mid d_{ki}^{\square}))\rightarrow(\square_{k}\square_{k},[(\square_{k}(d_{kj}^{\square}\mid d_{ki}^{\square}))\cdot((d_{kj}^{\square}\mid d_{ki}^{\square})\square_{k})])
\]
}and a commuting pasting square{\footnotesize{}
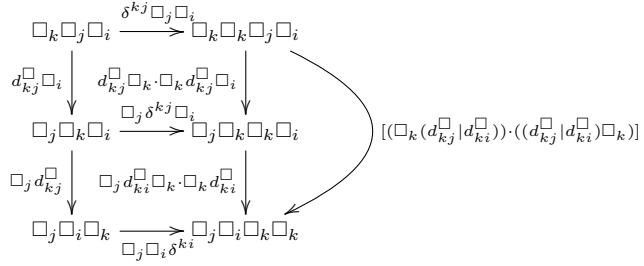
\[
\xymatrix{\square_{k}\square_{j}\square_{i}\ar[r]^{\delta^{kj}\square_{j}\square_{i}}\ar[d]_{d_{kj}^{\square}\square_{i}} & \square_{k}\square_{k}\square_{j}\square_{i}\ar[d]_{d_{kj}^{\square}\square_{k}\cdot\square_{k}d_{kj}^{\square}\square_{i}}\ar@/^{4pc}/[dd]^{[(\square_{k}(d_{kj}^{\square}\mid d_{ki}^{\square}))\cdot((d_{kj}^{\square}\mid d_{ki}^{\square})\square_{k})]}\\
\square_{j}\square_{k}\square_{i}\ar[d]_{\square_{j}d_{kj}^{\square}}\ar[r]^{\square_{j}\delta^{kj}\square_{i}} & \square_{j}\square_{k}\square_{k}\square_{i}\ar[d]_{\square_{j}d_{ki}^{\square}\square_{k}\cdot\square_{k}d_{ki}^{\square}}\\
\square_{j}\square_{i}\square_{k}\ar[r]_{\square_{j}\square_{i}\delta^{ki}} & \square_{j}\square_{i}\square_{k}\square_{k}
}
\]
}{\footnotesize \par}
\end{enumerate}
\item We had analogous diagrams for vertical compositions 
\[
\left(\frac{d_{li}^{\square}}{d_{ki}^{\square}}\right):\square_{l}\square_{k}\square_{i}\rightarrow\square_{i}\square_{l}\square_{k}
\]
seen as comonad functors 
\[
((C,\square_{k}\square_{i}),(\square_{l},\left(\frac{d_{li}^{\square}}{d_{ki}^{\square}}\right)))
\]
and given by $d_{li}^{\square}\square_{k}\cdot\square_{l}d_{ki}^{\square}$.
That is,

\begin{enumerate}
\item from {\scriptsize{}
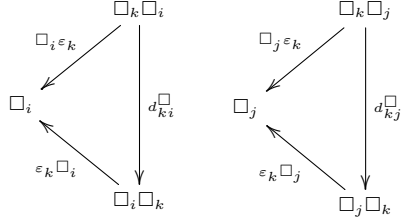
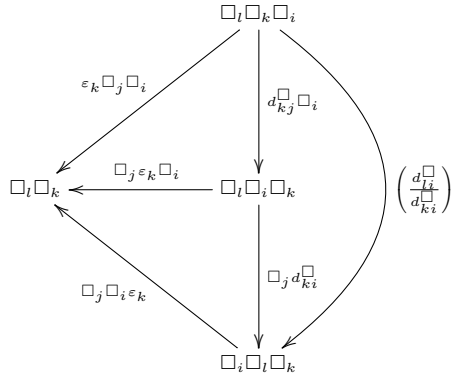
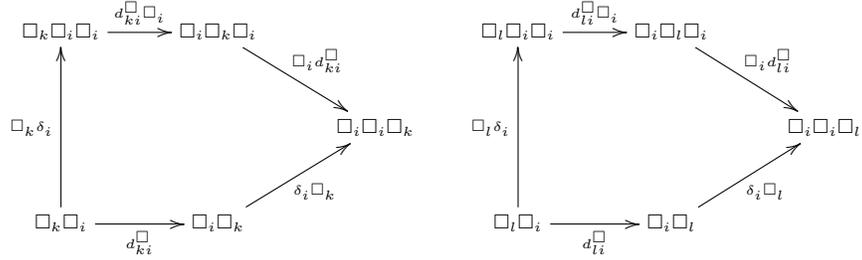
\[
\xymatrix{ & \square_{k}\square_{i}\ar[dl]_{\square_{i}\varepsilon_{k}}\ar[dd]^{d_{ki}^{\square}}\\
\square_{i}\\
 & \square_{i}\square_{k}\ar[ul]^{\varepsilon_{k}\square_{i}}
}
\qquad\xymatrix{ & \square_{k}\square_{j}\ar[dl]_{\square_{j}\varepsilon_{k}}\ar[dd]^{d_{kj}^{\square}}\\
\square_{j}\\
 & \square_{j}\square_{k}\ar[ul]^{\varepsilon_{k}\square_{j}}
}
\]
}we get {\scriptsize{}
\[
\xymatrix{ &  & \square_{l}\square_{k}\square_{i}\ar[ddll]_{\varepsilon_{k}\square_{j}\square_{i}}\ar[dd]^{d_{kj}^{\square}\square_{i}}\ar@/^{4pc}/[dddd]^{\left(\frac{d_{li}^{\square}}{d_{ki}^{\square}}\right)}\\
\\
\square_{l}\square_{k} &  & \square_{l}\square_{i}\square_{k}\ar[ll]_{\square_{j}\varepsilon_{k}\square_{i}}\ar[dd]^{\square_{j}d_{ki}^{\square}}\\
\\
 &  & \square_{i}\square_{l}\square_{k}\ar[lluu]^{\square_{j}\square_{i}\varepsilon_{k}}
}
\]
}{\scriptsize \par}
\item and from {\scriptsize{}
\[
\xymatrix{\square_{k}\square_{i}\square_{i}\ar[r]^{d_{ki}^{\square}\square_{i}} & \square_{i}\square_{k}\square_{i}\ar[dr]^{\square_{i}d_{ki}^{\square}}\\
 &  & \square_{i}\square_{i}\square_{k}\\
\square_{k}\square_{i}\ar[uu]^{\square_{k}\delta_{i}}\ar[r]_{d_{ki}^{\square}} & \square_{i}\square_{k}\ar[ru]_{\delta_{i}\square_{k}}
}
\qquad\xymatrix{\square_{l}\square_{i}\square_{i}\ar[r]^{d_{li}^{\square}\square_{i}} & \square_{i}\square_{l}\square_{i}\ar[dr]^{\square_{i}d_{li}^{\square}}\\
 &  & \square_{i}\square_{i}\square_{l}\\
\square_{l}\square_{i}\ar[uu]^{\square_{l}\delta_{i}}\ar[r]_{d_{li}^{\square}} & \square_{i}\square_{l}\ar[ru]_{\delta_{i}\square_{l}}
}
\]
}we can construct the following commuting diagram{\scriptsize{}
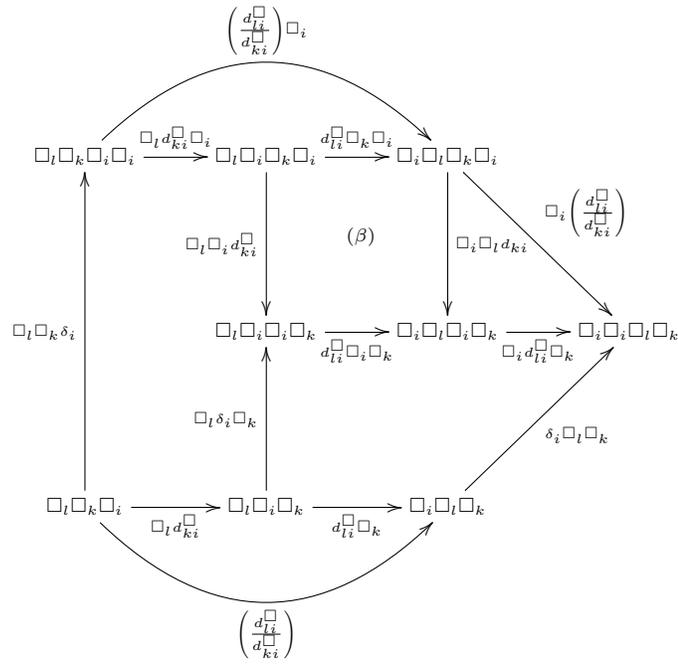
\[
\xymatrix{\square_{l}\square_{k}\square_{i}\square_{i}\ar[r]^{\square_{l}d_{ki}^{\square}\square_{i}}\ar@/^{3pc}/[rr]^{\left(\frac{d_{li}^{\square}}{d_{ki}^{\square}}\right)\square_{i}} & \square_{l}\square_{i}\square_{k}\square_{i}\ar[r]^{d_{li}^{\square}\square_{k}\square_{i}}\ar[dd]_{\square_{l}\square_{i}d_{ki}^{\square}} & \square_{i}\square_{l}\square_{k}\square_{i}\ar[dd]^{\square_{i}\square_{l}d_{ki}}\ar[rdd]^{\square_{i}\left(\frac{d_{li}^{\square}}{d_{ki}^{\square}}\right)}\ar@{}[ddl]|{\mathstrut\raisebox{2ex}{\textrm{ }\textrm{ }\ensuremath{(\beta)}\kern0.3em }}\\
\\
 & \square_{l}\square_{i}\square_{i}\square_{k}\ar[r]_{d_{li}^{\square}\square_{i}\square_{k}} & \square_{i}\square_{l}\square_{i}\square_{k}\ar[r]_{\square_{i}d_{li}^{\square}\square_{k}} & \square_{i}\square_{i}\square_{l}\square_{k}\\
\\
\square_{l}\square_{k}\square_{i}\ar[uuuu]^{\square_{l}\square_{k}\delta_{i}}\ar[r]_{\square_{l}d_{ki}^{\square}}\ar@/_{3pc}/[rr]_{\left(\frac{d_{li}^{\square}}{d_{ki}^{\square}}\right)} & \square_{l}\square_{i}\square_{k}\ar[uu]^{\square_{l}\delta_{i}\square_{k}}\ar[r]_{d_{li}^{\square}\square_{k}} & \square_{i}\square_{l}\square_{k}\ar[ruu]_{\delta_{i}\square_{l}\square_{k}}
}
\]
}for the second condition of a comonad functor where the square $(\beta)$
commutes by naturality of composition.
\end{enumerate}
For $\varepsilon$ and $\delta$ we compute as above.
\end{enumerate}
\end{proof}
We are in this way obtaining a Distributive Law for three modalities
in the form 
\[
\left(\frac{d_{N{}_{2}N{}_{1}}^{\square}}{d_{N_{3}N_{1}}^{\square}\mid d_{N_{3}}^{\square}}\right)=\left(\frac{d_{N_{2}N_{1}}^{\square}}{d_{N_{3}N_{1}}^{\square}}\mid d_{N{}_{3}}^{\square}\right):(N_{3}N_{2})N_{1}\rightarrow N_{1}(N_{3}N_{2})
\]

\section{The symmetric version}

To get the expressivity of the different Multimodal systems we need
all modalities interact. This is obtained after considering hypercubes
where a number of Distributive Laws exist in such a way that every
chain of modalities becomes a new modality in its own right. For that
we consider \emph{transposition functors} between the different multiple
sets underlying $DCmd_{\mathbf{n}}(\mathcal{C})$, as known from the
\emph{Symmetric Group Theory}, to the permutations among all axis
in the hypercube.
\begin{defn}
\label{symmetric}A \emph{multiple category of symmetric cubical type}
is a multiple category of cubical type with an assigned action of
the symmetric group $S_{n}$ on each set $X_{\mathbf{i}}$ for $\mathbf{i}$
a multi-index with length $n$ generated by \emph{transpositions}
$s_{i}:X_{\mathbf{i}}\rightarrow X_{\mathbf{i}}$ switching the $i-$
and $(i-1)-$ axis for $i=1,...,n-1$.\footnote{For $n<2$ we have no transpositions so we work from now on with $n\geq2$.}
\end{defn}
We denote by $S_{n}$ the n-symmetric group. Since $s_{i}$, as 2-cycles,
generate all elements in $S_{n}$ we have all permutations of coordinates
available to combine the different modalities. For example, for the
case of $n=2$ we can permute all axis in every square and every cube
giving rise to, for instance, to {\footnotesize{}
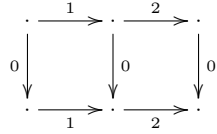
\[
\xymatrix{\cdot\ar[r]^{1}\ar[d]_{0} & \cdot\ar[r]^{2}\ar[d]^{0} & \cdot\ar[d]^{0}\\
\cdot\ar[r]_{1} & \cdot\ar[r]_{2} & \cdot
}
\]
}obtained by composition of coordinate systems 

{\footnotesize{}
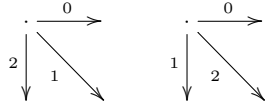
\[
\xymatrix{\cdot\ar[r]^{0}\ar[d]_{2}\ar[dr]_{1} & \mbox{}\\
\mbox{} & \mbox{}
}
\qquad\xymatrix{\cdot\ar[r]^{0}\ar[d]_{1}\ar[dr]_{2} & \mbox{}\\
\mbox{} & \mbox{}
}
\]
} 

\noindent in the direction of axis $1$. That is, in the notation
of \cite{Grandis-Par=0000E9} for \emph{oriented compositions} and
making use of sequential composition of modalities from Section 1:
\[
d_{01}^{\square}+_{1}d_{02}^{\square}=d_{0(21)}^{\square}
\]

For the calculation done we note also that, in the presence of transpositions,
we do not have just Distributive Laws in the form $d_{ji}^{\square}$
for $j\geq i$ but a Distributive Law for every pair of indices no
matter which one the greater is. When ordering the subindexes according
to the order in $\mathbf{n}$ we get an \emph{indexing normal form}.

We now show how transpositions interact with composition in different
directions:{\footnotesize{}
\[
\begin{array}{c}
\begin{array}{c}
s_{i}(a+_{i}b)=s_{i}(a)+_{i-1}s_{i}(b)\qquad s_{i}(a+_{i-1}b)=s_{i}(a)+_{i}s_{i}(b)\\
s_{i}(a+_{j}b)=s_{i}(a)+_{j}s_{i}(b)\quad for\;j\neq i-1,i
\end{array}\end{array}
\]
}that is, for $n=3$ we have three modalities $0,1,2$ and two transpositions
$s_{1},s_{2}$ such that {\footnotesize{}
\[
\begin{array}{c}
\begin{array}{c}
\begin{array}{c}
s_{1}(a+_{0}b)=s_{1}(a)+_{1}s_{1}(b)\\
s_{1}(a+_{1}b)=s_{1}(a)+_{0}s_{1}(b)\\
s_{1}(a+_{2}b)=s_{1}(a)+_{2}s_{1}(b)
\end{array}\end{array}\qquad\begin{array}{c}
\begin{array}{c}
s_{2}(a+_{0}b)=s_{2}(a)+_{0}s_{2}(b)\\
s_{2}(a+_{1}b)=s_{2}(a)+_{2}s_{2}(b)\\
s_{2}(a+_{2}b)=s_{2}(a)+_{1}s_{2}(b)
\end{array}\end{array}\end{array}
\]
}which means for example{\tiny{}
\[
s_{1}\left(\vcenter{\xymatrix{\cdot\ar[r]^{1}\ar[d]_{2} & \cdot\ar[r]^{1} & \cdot\ar[d]^{2}\ar@{}[dll]|{\mathstrut\raisebox{1ex}{\textrm{ }\textrm{ }\ensuremath{\textrm{ }\ensuremath{a+_{1}b}\kern0.3em }\kern0.3em }}\\
\cdot\ar[r]_{1} & \cdot\ar[r]_{1} & \cdot
}
}\right)=\vcenter{\xymatrix{\cdot\ar[r]^{0}\ar[d]_{2} & \cdot\ar[r]^{0} & \cdot\ar[d]^{2}\ar@{}[dll]|{\mathstrut\raisebox{1ex}{\textrm{ }\textrm{ }\ensuremath{s_{1}(a)+_{0}s_{1}(b)}\kern0.3em }}\\
\cdot\ar[r]_{0} & \cdot\ar[r]_{0} & \cdot
}
}=\left[s_{1}\left(\vcenter{\xymatrix{\cdot\ar[r]^{1}\ar[d]_{2} & \cdot\ar[d]^{2}\ar@{}[dl]|{\mathstrut\raisebox{1ex}{\textrm{ }\textrm{ }\ensuremath{a}\kern0.3em }}\\
\cdot\ar[r]_{1} & \cdot
}
}\right)\right]+_{0}\left[s_{1}\left(\vcenter{\xymatrix{\cdot\ar[r]^{1}\ar[d]_{2} & \cdot\ar[d]^{2}\ar@{}[dl]|{\mathstrut\raisebox{1ex}{\textrm{ }\textrm{ }\ensuremath{b}\kern0.3em }}\\
\cdot\ar[r]_{1} & \cdot
}
}\right)\right]
\]
}{\tiny \par}

That is:
\[
s_{1}(\square_{2}\square_{1}\square_{1}\rightarrow\square_{1}\square_{1}\square_{2})=\square_{2}\square_{0}\square_{0}\rightarrow\square_{0}\square_{0}\square_{2}
\]

\begin{defn}
Let $SDCmd_{\mathbf{n}}(\mathcal{C})$ be the multiple category of
symmetric cubical type generated by adding transpositions to $DCmd_{n}(\mathcal{C})$.
\end{defn}
Transpositions act over Distributive Laws in $DCmd_{\mathbf{n}}(\mathcal{C})$
by switching the comonads. Therefore, whenever $n=2$, we get from
a 2-cell as the one at left a 2-cell as the one at right: {\footnotesize{}
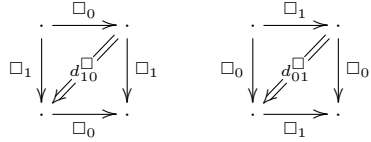
\[
\xymatrix{\cdot\ar[r]^{\square_{0}}\ar[d]_{\square_{1}} & \cdot\ar[d]^{\square_{1}}\ar@{=>}[dl]|{d_{10}^{\square}}\\
\cdot\ar[r]_{\square_{0}} & \cdot
}
\qquad\xymatrix{\cdot\ar[r]^{\square_{1}}\ar[d]_{\square_{0}} & \cdot\ar[d]^{\square_{0}}\ar@{=>}[dl]|{d_{01}^{\square}}\\
\cdot\ar[r]_{\square_{1}} & \cdot
}
\]
}by an action of the only transposition $s_{1}$. That is, since a
2-cell Distributive Law is a square, $s_{i}$ acting over a square
is the same thing as acting over a Distributive Law and we have $s_{1}(d_{10}^{\square})=d_{01}^{\square}$. 

In general:
\[
s_{k}(d_{ij}^{\square})=\begin{cases}
d_{(i-1)j}^{\square} & whenever\;i=k,j\neq k,k-1\\
d_{(i+1)j}^{\square} & whenever\;i=k-1,j\neq k,k-1\\
d_{i(j-1)}^{\square} & whenever\;j=k,i\neq k,k-1\\
d_{i(j+1)}^{\square} & whenever\;j=k-1,i\neq k,k-1\\
d_{ji}^{\square} & whenever\;j=k,i=k-1\\
d_{ij}^{\square} & otherwise
\end{cases}
\]

\noindent for $i,j,k\in\mathbf{n}$ such that $j\leq i$. Similarly,
for composed indexing we had
\[
s_{1}(d_{2(11)}^{\square})=d_{2(00)}^{\square}
\]

We now show how can one compose Distributive Laws in $SDCmd_{\mathbf{n}}(\mathcal{C})$
over a 2-category. They are, for different indices $\alpha,\gamma,\beta$
in $\mathbf{n}$ and $0<j<n-1$:{\footnotesize{}
\[
\begin{array}{c}
d_{\alpha\beta}^{\square}+_{(n-1)}d_{\gamma\beta}^{\square}=d_{(\gamma\alpha)\beta}^{\square}\qquad d_{\alpha\beta}^{\square}+_{0}d_{\alpha\gamma}^{\square}=d_{\alpha(\gamma\beta)}^{\square}\\
d_{\alpha\beta}^{\square}+_{j}d_{\alpha\beta}^{\square}=\begin{cases}
d_{(\alpha\alpha)\beta}^{\square} & for\;\partial_{j}^{-}(d_{\alpha\beta}^{\square}+_{j}d_{\alpha\beta}^{\square})=\beta\\
d_{\alpha(\beta\beta)}^{\square} & for\;\partial_{j}^{-}(d_{\alpha\beta}^{\square}+_{j}d_{\alpha\beta}^{\square})=\alpha
\end{cases}
\end{array}
\]
}{\footnotesize \par}

Taking into account (following \cite{Dosen-Petric}) that \emph{a
modality is a finite (possibly empty) sequence of the modal operators
of necessity and possibility,} to get a multimodal system as that
of \cite{Baldoni} in the style of cubical sets we point that every
level in this construction contains one more modality than the former
and correspond to each category $DCmd_{\mathbf{i}}(\mathcal{C})$.
They are connected through the face and degeneration functions whose
action erases or adds one modality respectively, while the empty modality
gives all objects that can be formed in a bicartesian closed category
as the analogues of conjunction, disjunction and implication. 

We have for example 
\[
e_{2}:DCmd_{\mathbf{i}\setminus\{2\}}(\mathcal{C})\rightarrow DCmd_{\mathbf{i}}(\mathcal{C})
\]
such that $e_{2}(NC)=N'NC$ for: 
\begin{itemize}
\item $C$ an object in the 2-category $\mathcal{C}$
\item $N=\square_{0},\square_{1}$ or a chain of possibly several $\square_{0}$
and $\square_{1}$ and 
\item $N'=\square_{0},\square_{1},\square_{2}$ or a chain of possibly $\square_{0}$,
$\square_{1}$ and $\square_{2}$. 
\end{itemize}

\section{The cubical categories $SDMnd(\mathcal{C})$}

We now turn to consider an n-dimensional multiple category of cubical
type by endowing a double category with $m$ arrows in different directions
which turn out to be monads. 
\begin{defn}
Let $DMnd_{\mathbf{2}}(\mathcal{C})$ be the full subcategory of $Q\mathcal{C}$
for which all 1-cells are monads and all diagonal 2-cells are Distributive
Laws between them. 

$DMnd_{\mathbf{2}}(\mathcal{C})$ is a double category which is also
endowed with horizontal maps $aC:M_{2}M_{1}C\rightarrow M_{1}M_{2}C$
in a square{\footnotesize{}
\[
\xymatrix{M_{2}M_{1}C\ar[r]^{aC}\ar[d]_{M_{2}M_{1}u} & M_{1}M_{2}C\ar[d]^{M_{1}M_{2}u}\ar@{=>}[dl]|{au}\\
M_{2}M_{1}C\ar[r]_{aC} & M_{1}M_{2}C
}
\]
}for $C$ an object and $u:C\rightarrow C$ a 1-endocell in $\mathcal{C}$.
It should be noticed that, in this case, the diagonal 2-cells are
written $d_{12}^{\lozenge}$ (switching the subindices from the case
of $d_{21}^{\square}$ for the same square).
\end{defn}
From the double category considered, we can now define $DMnd_{\mathbf{m}}(\mathcal{C})$
as the m-dimensional multiple category of cubical type of Distributive
Laws between monads over a 2-category. $DMnd_{\mathbf{2}}(\mathcal{C})$
can be seen as the image of $Mnd(Mnd(\mathcal{C}))$ through the functor
of quintets $Q$. We begin by showing how can one compose Distributive
Laws between monads in a cell {\footnotesize{}
\[
\xymatrix{\cdot\ar[r]^{M_{1}}\ar[d]_{M_{2}} & \cdot\ar[d]\ar[r]^{M'_{1}}\ar@{=>}[dl]|{d_{M_{1}M_{2}}^{\lozenge}} & \cdot\ar[d]^{M_{2}}\ar@{=>}[dl]|{d_{M'_{1}M{}_{2}}^{\lozenge}}\\
\cdot\ar[d]_{M'_{2}}\ar[r] & \cdot\ar[d]\ar[r]\ar@{=>}[dl]|{d_{M{}_{1}M'_{2}}^{\lozenge}} & \cdot\ar[d]^{M_{2}'}\ar@{=>}[dl]|{d_{M'_{1}M'_{2}}^{\lozenge}}\\
\cdot\ar[r]_{M_{1}} & \cdot\ar[r]_{M'_{1}} & \cdot
}
\]
}{\footnotesize \par}

\noindent we can compose these Distributive Laws horizontal and vertically
denoted respectively 
\[
(d_{M_{1}M_{2}}^{\lozenge}\mid d_{M'_{1}M_{2}}^{\lozenge})=M_{1}'d_{M_{1}M_{2}}^{\lozenge}\cdot d_{M'_{1}M_{2}}^{\lozenge}M_{1}\qquad\left(\frac{d_{M{}_{1}M_{2}}^{\lozenge}}{d_{M_{1}M'_{2}}^{\lozenge}}\right)=d_{M_{1}M'_{2}}^{\lozenge}M{}_{2}\cdot M'_{2}d_{M{}_{1}M_{2}}^{\lozenge}
\]
where both compositions are strict thanks to the 2-categorical structure
of $\mathcal{C}$. 

The \emph{four-middle interchange rule} has the form:
\[
\left(\frac{d_{M_{1}M_{2}}^{\lozenge}\mid d_{M'_{1}M_{2}}^{\lozenge}}{d_{M_{1}M'_{2}}^{\lozenge}\mid d_{M'_{1}M'_{2}}^{\lozenge}}\right)=(\frac{d_{M{}_{1}M_{2}}^{\lozenge}}{d_{M_{1}M'_{2}}^{\lozenge}}\mid\frac{d_{M{}_{2}M'_{1}}^{\lozenge}}{d_{M'_{2}M'_{1}}^{\lozenge}})
\]
To prove that horizontal and vertical compositions are again Distributive
Laws for monads we have just to dualize the diagrams from the 2-categorical
definition in Section 3.
\begin{lem}
Every composition of 2-cells of the form $d^{\lozenge}$ in $DMnd_{\mathbf{m}}(\mathcal{C})$
is a Distributive Law.
\end{lem}
By adding transpositions we had, as in Section 4, a multiple category
of symmetric cubical type $SDMnd_{\mathbf{m}}(\mathcal{C})$. We have
an analogous definition for \emph{transposed Distributive Laws}:

\[
s_{k}(d_{ij}^{\lozenge})=\begin{cases}
d_{(i-1)j}^{\lozenge} & whenever\;i=k,j\neq k,k-1\\
d_{(i+1)j}^{\lozenge} & whenever\;i=k-1,j\neq k,k-1\\
d_{i(j-1)}^{\lozenge} & whenever\;j=k,i\neq k,k-1\\
d_{i(j+1)}^{\lozenge} & whenever\;j=k-1,i\neq k,k-1\\
d_{ji}^{\lozenge} & whenever\;j=k,i=k-1\\
d_{ij}^{\lozenge} & otherwise
\end{cases}
\]

\noindent for $i,j,k\in\mathbf{m}$ such that $j\leq i$. Similarly,
for composed indexing we had
\[
s_{1}(d_{2(11)}^{\lozenge})=d_{2(00)}^{\lozenge}
\]

The possibility operators, denoted by $\diamondsuit$, in a sense
dual to necessity, are performed in our setting by dualizing most
of the structure given up to now. With the same indexing language
$Mod$ a set of \emph{incestual axioms} in the form $G^{a,b,c,d}:\lozenge_{a}\square_{b}A\longrightarrow\square_{c}\lozenge_{d}A$
with $b=c=\epsilon$ will develop the fragment of $G^{a,\epsilon,\epsilon,d}$
for which we consider just diamonds.

\section{The cubical categories $SEnt(\mathcal{C})$}

We now consider an $n$-dimensional multiple category of cubical type
by endowing a double category with $n$ arrows in different directions
corresponding to $\left\lceil \nicefrac{n}{2}\right\rceil $ monads
and $\left\lceil \nicefrac{n+1}{2}\right\rceil $ comonads. We need
then to make use of \emph{entwining Laws} all over them in the form
of the following definitions (see \cite{Power-Watanabe}).
\begin{defn}
\label{ent1}Given a monad $M$ and a comonad $N$ over a 2-category
$\mathcal{C}$ we say that the 2-cell $e:MN\rightarrow NM$ is an
\emph{entwining Law of $M$ over $N$} or a \emph{mixed Distributive
Law from a monad $M$ to a comonad $N$} if the following diagrams
commute {\footnotesize{}
\[
\xymatrix{MMN\ar[rr]^{\mu N}\ar[d]_{Me} &  & MN\ar[d]^{e}\\
MNM\ar[r]_{eM} & NMM\ar[r]_{N\mu} & NM
}
\qquad\xymatrix{MN\ar[r]^{\mu N}\ar[d]_{Me} & MNN\ar[r]^{eN} & NMN\ar[d]^{Ne}\\
NM\ar[rr]_{\delta M} &  & NNM
}
\]
\[
\xymatrix{N\ar[r]^{\eta N}\ar[rd]_{N\eta} & MN\ar[d]^{e}\\
 & NM
}
\qquad\xymatrix{MN\ar[r]^{M\varepsilon}\ar[d]_{e} & M\\
NM\ar[ru]_{\varepsilon M}
}
\]
}{\footnotesize \par}
\end{defn}
\noindent By dualizing this definition we obtain the definition of
an \emph{entwining Law of a comonad $N$ over a monad $M$} or a \emph{mixed
Distributive Law from a comonad $N$ to a monad $M$. }We denote them
respectively $d^{\lozenge\square}$ and $d^{\square\lozenge}$.
\begin{defn}
Let $Ent{}_{\mathbf{2}}(\mathcal{C})$ be the full subcategory of
$Q\mathcal{C}$ for which all horizontal 1-cells are comonads, all
vertical 1-cells are monads and all diagonal 2-cells are Distributive
Laws $d^{\lozenge\square}$ between them.
\end{defn}

Entwining Laws of type $d^{\lozenge\square}$ in it are diagonal 2-cells
giving rise to double categories which are also endowed with horizontal
maps $d^{\lozenge\square}C:MNC\rightarrow NMC$ in squares{\footnotesize{}
\[
\xymatrix{MNC\ar[r]^{d^{\lozenge\square}C}\ar[d]_{MNu} & NMC\ar[d]^{NMu}\ar@{=>}[dl]|{d^{\lozenge\square}C}\\
MNC\ar[r]_{d^{\lozenge\square}C} & NMC
}
\]
}for $C$ an object, $u:C\rightarrow C$ a 1-endocell in $\mathcal{C}$
and $M$ and $N$ a monad and a comonad over $C$ respectively.

From the double category considered we can define $Ent_{\mathbf{n}}(\mathcal{C})$
as the multiple category of cubical type whose nodes are the objects
in a 2-category and whose underlying category is bicartesian closed.
We establish the following conventions: axis in it indexed by odd
numbers in $\mathbf{n}$ are monads, axis indexed by even numbers
in $\mathbf{n}$ are comonads and the 2-cells between them are denoted
$d_{ij}$ and, for $i>j$, identified as:
\[
d_{ij}=\begin{cases}
d_{ij}^{\square} & whenever\;i,j\;are\;even\\
d_{ij}^{\lozenge\square} & whenever\;i\;is\;odd,\;j\;is\;even\\
d_{ij}^{\square\lozenge} & whenever\;i\;is\;even,\;j\;is\;odd\\
d_{ij}^{\lozenge} & whenever\;i,j\;are\;odd
\end{cases}
\]
for $i,j\in\mathbf{n}$ such that $i\geq j$.\footnote{Notice that we do not still have 2-cells in the form $d_{ij}$ with
$j>i$.}

Where we denote the mixed Distributive Laws as 
\[
d_{ij}^{\lozenge\square}:\lozenge_{i}\square_{j}\longrightarrow\square_{j}\lozenge_{i}\qquad d_{ij}^{\square\lozenge}:\square_{i}\lozenge_{j}\longrightarrow\lozenge_{j}\square_{i}
\]

For instance, the different Distributive Laws for our system in the
model of four axis {\footnotesize{}
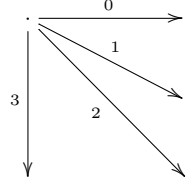
\[
\xymatrix{\cdot\ar[rr]^{0}\ar[drr]^{1}\ar[ddrr]_{2}\ar[dd]_{3} & \mbox{} & \mbox{}\\
\mbox{} & \mbox{} & \mbox{}\\
\mbox{} &  & \mbox{}
}
\]
}are $d_{10}^{\lozenge\square},d_{20}^{\square},d_{21}^{\square\lozenge},d_{30}^{\lozenge\square},d_{31}^{\lozenge}$
and $d_{32}^{\lozenge\square}$.

We give a description of $Ent_{\mathbf{2}}(\mathcal{C})$ by showing
composition of several Distributive Laws {\footnotesize{}
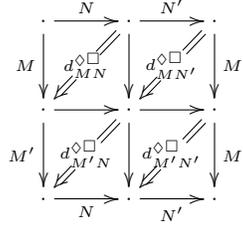
\[
\xymatrix{\cdot\ar[r]^{N}\ar[d]_{M} & \cdot\ar[d]\ar[r]^{N'}\ar@{=>}[dl]|{d_{MN}^{\lozenge\square}} & \cdot\ar[d]^{M}\ar@{=>}[dl]|{d_{MN'}^{\lozenge\square}}\\
\cdot\ar[d]_{M'}\ar[r] & \cdot\ar[d]\ar[r]\ar@{=>}[dl]|{d_{M'N}^{\lozenge\square}} & \cdot\ar[d]^{M'}\ar@{=>}[dl]|{d_{M'N'}^{\lozenge\square}}\\
\cdot\ar[r]_{N} & \cdot\ar[r]_{N'} & \cdot
}
\]
}{\footnotesize \par}

\noindent where $N$ are comonads, $M$ are monads and $d^{\lozenge\square}$
are entwining Laws with the appropriate subindexes. We compose the
entwining Laws into this square horizontal and vertically denoted
respectively 
\[
(d_{MN}^{\lozenge\square}\mid d_{MN'}^{\lozenge\square})=N'd_{MN}^{\lozenge\square}\cdot d_{MN'}^{\lozenge\square}N\qquad\left(\frac{d_{MN}^{\lozenge\square}}{d_{M'N}^{\lozenge\square}}\right)=d_{M'N}^{\lozenge\square}M\cdot M'd_{MN}^{\lozenge\square}
\]
where both compositions are strict thanks again to the 2-categorical
structure of $\mathcal{C}$. 

The \emph{four-middle interchange rule} has the form:
\[
\left(\frac{d_{MN}^{\lozenge\square}\mid d_{MN'}^{\lozenge\square}}{d_{NM'}^{\lozenge\square}\mid d_{M'N'}^{\lozenge\square}}\right)=(\frac{d_{MN}^{\lozenge\square}}{d_{NM'}^{\lozenge\square}}\mid\frac{d_{MN'}^{\lozenge\square}}{d_{M'N'}^{\lozenge\square}})
\]

\begin{lem}
Every composition of Distributive Laws of the form $d^{\mathcal{M}}$
in $Ent_{\mathbf{n}}(\mathcal{C})$ is a Distributive Law in the form
$d^{\mathcal{M}}$ for $\mathcal{M}=\square,\lozenge,\lozenge\square\textrm{ or }\square\lozenge$.
\end{lem}
\begin{proof}
We consider for example two entwining Laws $d_{j_{1}i_{1}}^{\lozenge\square}$
and $d_{j_{1}i_{2}}^{\lozenge\square}$ with $i_{1},j_{1},i_{2},j_{2}\in\mathbf{n}$
for which the following diagrams commute:{\footnotesize{}
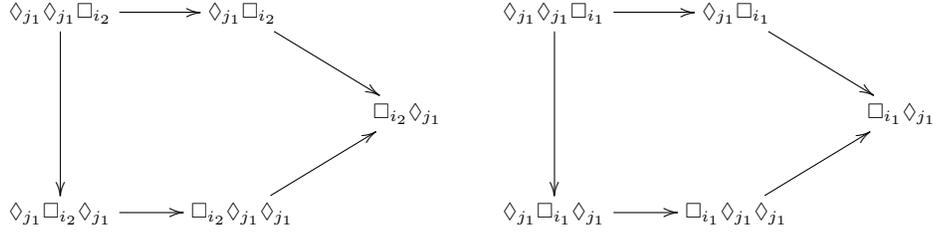
\[
\xymatrix{\lozenge_{j_{1}}\lozenge_{j_{1}}\square_{i_{2}}\ar[r]\ar[dd] & \lozenge_{j_{1}}\square_{i_{2}}\ar[rd]\\
 &  & \square_{i_{2}}\lozenge_{j_{1}}\\
\lozenge_{j_{1}}\square_{i_{2}}\lozenge_{j_{1}}\ar[r] & \square_{i_{2}}\lozenge_{j_{1}}\lozenge_{j_{1}}\ar[ur]
}
\qquad\xymatrix{\lozenge_{j_{1}}\lozenge_{j_{1}}\square_{i_{1}}\ar[r]\ar[dd] & \lozenge_{j_{1}}\square_{i_{1}}\ar[rd]\\
 &  & \square_{i_{1}}\lozenge_{j_{1}}\\
\lozenge_{j_{1}}\square_{i_{1}}\lozenge_{j_{1}}\ar[r] & \square_{i_{1}}\lozenge_{j_{1}}\lozenge_{j_{1}}\ar[ur]
}
\]
}We then have two commutative diagrams{\scriptsize{}
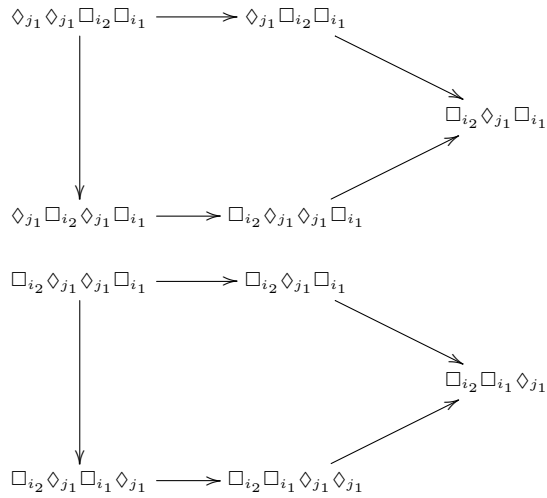
\[
\xymatrix{\lozenge_{j_{1}}\lozenge_{j_{1}}\square_{i_{2}}\square_{i_{1}}\ar[r]\ar[dd] & \lozenge_{j_{1}}\square_{i_{2}}\square_{i_{1}}\ar[rd]\\
 &  & \square_{i_{2}}\lozenge_{j_{1}}\square_{i_{1}}\\
\lozenge_{j_{1}}\square_{i_{2}}\lozenge_{j_{1}}\square_{i_{1}}\ar[r] & \square_{i_{2}}\lozenge_{j_{1}}\lozenge_{j_{1}}\square_{i_{1}}\ar[ur]
}
\]
\[
\xymatrix{\square_{i_{2}}\lozenge_{j_{1}}\lozenge_{j_{1}}\square_{i_{1}}\ar[r]\ar[dd] & \square_{i_{2}}\lozenge_{j_{1}}\square_{i_{1}}\ar[rd]\\
 &  & \square_{i_{2}}\square_{i_{1}}\lozenge_{j_{1}}\\
\square_{i_{2}}\lozenge_{j_{1}}\square_{i_{1}}\lozenge_{j_{1}}\ar[r] & \square_{i_{2}}\square_{i_{1}}\lozenge_{j_{1}}\lozenge_{j_{1}}\ar[ur]
}
\]
}by considering the products $d_{j_{1}i_{2}}^{\lozenge\square}\square_{i_{1}}$
and $\square_{i_{2}}d_{j_{1}i_{1}}^{\lozenge\square}$. 

Now we paste them to get{\tiny{}
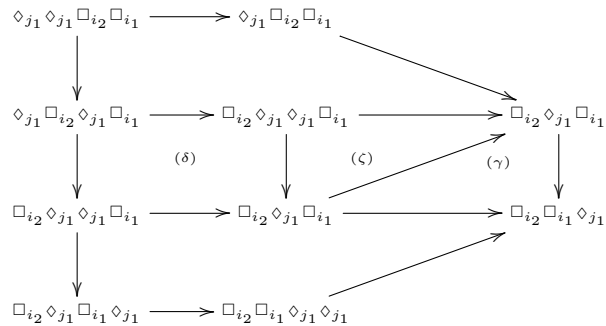
\[
\xymatrix{\lozenge_{j_{1}}\lozenge_{j_{1}}\square_{i_{2}}\square_{i_{1}}\ar[r]\ar[d] & \lozenge_{j_{1}}\square_{i_{2}}\square_{i_{1}}\ar[rrd]\\
\lozenge_{j_{1}}\square_{i_{2}}\lozenge_{j_{1}}\square_{i_{1}}\ar[r]\ar[d]\ar@{}[dr]|{\mathstrut\raisebox{2ex}{\textrm{ }\textrm{ }\ensuremath{(\delta)}\kern0.3em }} & \square_{i_{2}}\lozenge_{j_{1}}\lozenge_{j_{1}}\square_{i_{1}}\ar[rr]\ar[d]\ar@{}[dr]|{\mathstrut\raisebox{2ex}{\textrm{ }\textrm{ }\ensuremath{(\zeta)}\kern0.3em }} &  & \square_{i_{2}}\lozenge_{j_{1}}\square_{i_{1}}\ar@{}[dl]|{\mathstrut\raisebox{.5ex}{\textrm{ }\textrm{ }\ensuremath{(\gamma)}\kern0.3em }}\ar[d]\\
\square_{i_{2}}\lozenge_{j_{1}}\lozenge_{j_{1}}\square_{i_{1}}\ar[r]\ar[d] & \square_{i_{2}}\lozenge_{j_{1}}\square_{i_{1}}\ar[urr]\ar[rr] &  & \square_{i_{2}}\square_{i_{1}}\lozenge_{j_{1}}\\
\square_{i_{2}}\lozenge_{j_{1}}\square_{i_{1}}\lozenge_{j_{1}}\ar[r] & \square_{i_{2}}\square_{i_{1}}\lozenge_{j_{1}}\lozenge_{j_{1}}\ar[urr]
}
\]
}where diagrams $(\delta)$, $(\zeta)$ and $(\gamma)$ commute trivially. 

With which we deduce that $(d_{j_{1}i_{1}}^{\lozenge\square}\mid d_{j_{1}i_{2}}^{\lozenge\square})$
behaves as an entwining Law according to Definition \ref{ent1}.
\end{proof}
By adding transpositions we have, as in Section 4, a multiple category
of symmetric cubical type $SEnt_{\mathbf{n}}(\mathcal{C})$ whose
Distributive Laws are for $i,j\in\mathbf{n}$:\footnote{$SEnt_{\mathbf{2}}(\mathcal{C})$ contains both the images of $Mnd(Cmd(\mathcal{C}))$
and $Cmd(Mnd(\mathcal{C}))$ through the functor of quintets $Q$
as well as the images of $Mnd(Mnd(\mathcal{C}))$ and $Cmd(Cmd(\mathcal{C}))$.} 
\[
d_{ij}=\begin{cases}
d_{ij}^{\square} & whenever\;i,j\;are\;even\\
d_{ij}^{\lozenge\square} & whenever\;i\;is\;odd,\;j\;is\;even\\
d_{ij}^{\square\lozenge} & whenever\;i\;is\;even,\;j\;is\;odd\\
d_{ij}^{\lozenge} & whenever\;i,j\;are\;odd
\end{cases}
\]

\section{Sets of axioms}

We list sets of incestual axioms that can be performed in the different
multiple categories of (symmetric and non-symmetric) cubical type.
For that we make use of the terminology given in \cite{Baldoni}.

{\footnotesize{}}%
\noindent\doublebox{\begin{minipage}[t]{1\columnwidth - 2\fboxsep - 7.5\fboxrule - 1pt}%
\centering \begin{center}
\textbf{\footnotesize{}Incestual interaction axioms generated in $DCmd(\mathcal{C})$
with $a=d=\epsilon$ for $G^{a,b,c,d}$ }
\par\end{center}{\footnotesize \par}
\begin{enumerate}
\item \emph{\footnotesize{}Reflexivity axiom}{\footnotesize{} for $G^{\epsilon,i,\epsilon,\epsilon}$:
$\square_{i}A\rightarrow A$}{\footnotesize \par}
\item \emph{\footnotesize{}Transitivity axiom}{\footnotesize{} for $G^{\epsilon,i,(i;i),\epsilon}$:
$\square_{i}A\rightarrow\square_{i}\square_{i}A$}{\footnotesize \par}
\item \emph{\footnotesize{}Restricted Persistency axiom}{\footnotesize{}
for $G^{\epsilon,(i;j),(j;i),\epsilon}$: $\square_{i}\square_{j}A\rightarrow\square_{j}\square_{i}A$.}{\footnotesize \par}
\end{enumerate}
\end{minipage}}{\footnotesize \par}

{\footnotesize{}}%
\noindent\doublebox{\begin{minipage}[t]{1\columnwidth - 2\fboxsep - 7.5\fboxrule - 1pt}%
\centering \begin{center}
\textbf{\footnotesize{}Incestual interaction axioms generated in $SDCmd(\mathcal{C})$
with $a=d=\epsilon$ for $G^{a,b,c,d}$}
\par\end{center}{\footnotesize \par}
\begin{enumerate}
\item \emph{\footnotesize{}General Persistency axiom}{\footnotesize{} for
$G^{\epsilon,(i;j),(j;i),\epsilon}$: $\square_{i}\square_{j}A\rightarrow\square_{j}\square_{i}A$}{\footnotesize \par}
\item \emph{\footnotesize{}Composition axiom}{\footnotesize{} for $G^{\epsilon,(j;i),i,\epsilon}$:
$\square_{j}A\square_{i}A\rightarrow\square_{i}A$.}{\footnotesize \par}
\end{enumerate}
\end{minipage}}{\footnotesize \par}

{\footnotesize{}}%
\noindent\doublebox{\begin{minipage}[t]{1\columnwidth - 2\fboxsep - 7.5\fboxrule - 1pt}%
\centering \begin{center}
\textbf{\footnotesize{}Incestual interaction axioms generated in $DMnd(\mathcal{C})$
with $b=c=\epsilon$ for $G^{a,b,c,d}$}
\par\end{center}{\footnotesize \par}
\begin{enumerate}
\item \emph{\footnotesize{}Reflexivity axiom}{\footnotesize{} for $G^{j,\epsilon,\epsilon,\epsilon}$:
$A\rightarrow\lozenge_{j}A$ }{\footnotesize \par}
\item \emph{\footnotesize{}Transitivity axiom}{\footnotesize{} for $G^{j,\epsilon,\epsilon,(j;j)}$:
$\lozenge_{j}\lozenge_{j}A\rightarrow\lozenge_{j}A$}{\footnotesize \par}
\item \emph{\footnotesize{}Restricted Persistency axiom}{\footnotesize{}
for $G^{(j;i),\epsilon,\epsilon,(i;j)}$: $\lozenge_{i}\lozenge_{j}A\rightarrow\lozenge_{j}\lozenge_{i}A$.}{\footnotesize \par}
\end{enumerate}
\end{minipage}}{\footnotesize \par}

{\footnotesize{}}%
\noindent\doublebox{\begin{minipage}[t]{1\columnwidth - 2\fboxsep - 7.5\fboxrule - 1pt}%
\centering \begin{center}
\textbf{\footnotesize{}Incestual interaction axioms generated in $SDMnd(\mathcal{C})$
with $b=c=\epsilon$ for $G^{a,b,c,d}$}
\par\end{center}{\footnotesize \par}
\begin{enumerate}
\item \emph{\footnotesize{}General Persistency axiom}{\footnotesize{} for
$G^{(j;i),\epsilon,\epsilon,(i;j)}$: $\lozenge_{i}\lozenge_{j}A\rightarrow\lozenge_{j}\lozenge_{i}A$}{\footnotesize \par}
\item \emph{\footnotesize{}Composition axiom}{\footnotesize{} for $G^{(j;i),\epsilon,\epsilon,i}$:
$\lozenge_{j}A\rightarrow\lozenge_{i}\lozenge_{j}A$.}{\footnotesize \par}
\end{enumerate}
\end{minipage}}{\footnotesize \par}

{\footnotesize{}}%
\noindent\doublebox{\begin{minipage}[t]{1\columnwidth - 2\fboxsep - 7.5\fboxrule - 1pt}%
\centering \begin{center}
\textbf{\footnotesize{}Incestual interaction axioms generated in $Ent(\mathcal{C})$
for $G^{a,b,c,d}$}
\par\end{center}{\footnotesize \par}
\begin{enumerate}
\item \emph{\footnotesize{}Seriality axiom}{\footnotesize{} for $G^{\epsilon,i,\epsilon,j}$:
$\square_{i}A\rightarrow\lozenge_{j}A$}{\footnotesize \par}
\item \emph{\footnotesize{}Composition axiom}{\footnotesize{} for $G^{j,i,i,j}$:
$\lozenge_{j}\square_{i}A\rightarrow\square_{i}\lozenge_{j}A$}{\footnotesize \par}
\item {\footnotesize{}Axiom for $G^{j,i,\epsilon,j}$: $\lozenge_{j}\square_{i}A\rightarrow\lozenge_{j}A$.}{\footnotesize \par}
\end{enumerate}
\end{minipage}}{\footnotesize \par}

{\footnotesize{}}%
\noindent\doublebox{\begin{minipage}[t]{1\columnwidth - 2\fboxsep - 7.5\fboxrule - 1pt}%
\centering \begin{center}
\textbf{\footnotesize{}Incestual interaction axioms generated in $SEnt(\mathcal{C})$
not in the form $G^{a,b,c,d}$}
\par\end{center}{\footnotesize \par}
\begin{enumerate}
\item \emph{\footnotesize{}McKinsey axiom}{\footnotesize{}: $\square_{i}\lozenge_{j}A\rightarrow\lozenge_{j}\square_{i}A$}{\footnotesize \par}
\item {\footnotesize{}$\square_{i}A\rightarrow\lozenge_{j}\square_{i}A$.}{\footnotesize \par}
\end{enumerate}
\end{minipage}}{\footnotesize \par}

\end{document}